\documentclass[11pt,reqno]{amsart}

\usepackage[utf8]{inputenc}
\usepackage[T1]{fontenc}
\usepackage{lmodern}
\usepackage{amsmath,amssymb,amsthm,amscd,mathtools}
\usepackage[mathscr]{euscript}
\usepackage{mathrsfs}
\usepackage{enumitem}
\usepackage{xcolor}
\usepackage[colorlinks=true,linkcolor=blue,citecolor=black,urlcolor=black]{hyperref}
\usepackage{aliascnt}
\usepackage[nameinlink,noabbrev]{cleveref}

\theoremstyle{plain}
\newtheorem{theorem}{Theorem}[section]

\newaliascnt{lemma}{theorem}
\newtheorem{lemma}[lemma]{Lemma}
\aliascntresetthe{lemma}
\crefname{lemma}{lemma}{lemmas}
\Crefname{lemma}{Lemma}{Lemmas}

\newaliascnt{proposition}{theorem}
\newtheorem{proposition}[proposition]{Proposition}
\aliascntresetthe{proposition}
\crefname{proposition}{proposition}{propositions}
\Crefname{proposition}{Proposition}{Propositions}

\newaliascnt{corollary}{theorem}
\newtheorem{corollary}[corollary]{Corollary}
\aliascntresetthe{corollary}
\crefname{corollary}{corollary}{corollaries}
\Crefname{corollary}{Corollary}{Corollaries}

\theoremstyle{definition}
\newaliascnt{definition}{theorem}
\newtheorem{definition}[definition]{Definition}
\aliascntresetthe{definition}
\crefname{definition}{definition}{definitions}
\Crefname{definition}{Definition}{Definitions}

\newaliascnt{remark}{theorem}
\newtheorem{remark}[remark]{Remark}
\aliascntresetthe{remark}
\crefname{remark}{remark}{remarks}
\Crefname{remark}{Remark}{Remarks}

\newaliascnt{question}{theorem}

\aliascntresetthe{question}
\crefname{question}{question}{questions}
\Crefname{question}{Question}{Questions}

\newcommand{\N}{\mathbb N}
\newcommand{\R}{\mathbb R}
\newcommand{\C}{\mathbb C}

\newcommand{\norm}[1]{ \| #1 \| }

\newcommand{\restr}[2]{\left.#1\right|_{#2}}

\subjclass[2020]{Primary 46B03; Secondary 46B20, 46E15, 47B38}

\keywords{primary Banach space, primary factorisation property, vector-valued continuous functions, representing measures, compact metrizable spaces}

\begin{document}

\title{The Uniform Primary Factorisation Property for $C(K,E)$}

\author[A. Acuaviva]{Antonio Acuaviva}
\address{School of Mathematical Sciences,
Fylde College,
Lancaster University,
LA1 4YF,
United Kingdom} \email{ahacua@gmail.com}

\author[P. Acuaviva]{Pablo Acuaviva}
\address{Institute of Computer Science,
University of Bern,
Neubrückstrasse 10,
3012 Bern,
Switzerland} \email{pablohacuaviva@gmail.com}

\date{\today}

\begin{abstract}
    Let $K$ be an uncountable compact metrizable space, and $E$ be a separable Banach space with the uniform primary factorisation property (UPFP) and containing no isomorphic copy of $c_0$. We prove that $C(K, E)$ also has the UPFP. As a consequence, we obtain the primariness of the bi-parameter spaces $C([0,1],\ell_p)$ and $C([0,1], L_p)$ for $1 \leq p<\infty$.
\end{abstract}

\maketitle

\tableofcontents

\bigskip
\section{Introduction and organisation}\label{sec:introduction}

A Banach space $X$ is said to be \emph{primary} if, whenever $P\colon X\to X$ is a projection, either $PX$ or $(I_X-P)X$ is isomorphic to $X$. The primariness problem for classical Banach spaces has played an important role in the isomorphic theory of Banach spaces. 

A key emerging property in the study of primariness is the (uniform) primary factorisation property (U)PFP; see \Cref{def:upfp}. Roughly speaking, the PFP asks that, for every $T\in\mathscr{B}(X)$, the identity on $X$ factors through either $T$ or $I_X-T$, while the UPFP asks for uniform control of the factorisation constants. This point of view was isolated in the work of Acuaviva and Kania \cite{AcuavivaKania2026} as a way of turning scalar primariness arguments into Banach-valued ones: scalar lower-bound alternatives are replaced by uniform factorisation alternatives.

There is now a substantial literature on primariness questions for bi-parameter Banach spaces, that is, spaces built from two classical structures, such as $\ell_p(\ell_q)$, $L_p(L_q)$, $\ell_p(L_q)$ and $L_p(\ell_q)$. Many of these spaces have been treated by methods based on bases, martingales, Haar systems, or block decompositions; see, for instance, the work of Capon \cite{Capon1980,Capon1980b,Capon1982PLMS,Capon1983}, Wark \cite{Wark2007}, M\"uller \cite{Muller2012}, and the recent results of Lechner--Motakis--M\"uller--Schlumprecht \cite{LMMSS2022}.

Since its introduction, the UPFP has proved useful for primariness questions in bi-parameter settings, as it packages factorisation dichotomies into a form that is well suited to permanence arguments, allowing, in some circumstances, to lift scalar-valued proofs to the bi-parameter case. The first-named author used this point of view to prove the primariness of spaces of the form $\ell_p(C(K))$ for $1\leq p\leq\infty$ \cite{Acuaviva2026PrimarinesslpCK}. In subsequent work, related UPFP transfer principles were used to recover the primariness of $\ell_\infty(L_p)$ and $c_0(L_p)$ for $1\leq p<\infty$ \cite{Acuaviva2026PrimarinessSums}.

Our main result continues this line of work in the setting of vector-valued continuous function spaces $C(K,E)$.  We prove the following permanence theorem for the UPFP.

\begin{theorem}\label{thm:transfer}
Let $K$ be an uncountable compact metrizable space and $E$ be a separable Banach space containing no isomorphic copy of $c_0$. If $E$ has the UPFP, then $C(K, E)$ has the UPFP.
\end{theorem}

In particular, this gives new primariness results for bi-parameter spaces involving $C[0,1]$. This answers some of the questions of Lechner--Motakis--M\"uller--Schlumprecht \cite{LMMSS2022} concerning primariness of mixed classical Banach spaces involving $C(K)$ spaces, for $K$ metric and compact. We obtain the following.

\begin{corollary}\label{cor:ellp-Lp}
For $1\leq p<\infty$, the spaces $C([0,1],\ell_p)$ and $C([0,1],L_p)$ are primary.
\end{corollary}

\begin{proof}
    For $1\leq p<\infty$, the spaces $\ell_p$ and $L_p$ are separable and contain no isomorphic copy of $c_0$: by Sobczyk's theorem, it is enough to rule out complemented copies of $c_0$, which is immediate from reflexivity when $1<p<\infty$ and from weak sequential completeness when $p=1$.

    Moreover, $\ell_p$ has the UPFP by \cite[Section 3]{AcuavivaKania2026}. For $L_p$, the required quantitative factorisation statement can be extracted from the primariness argument of Alspach, Enflo and Odell \cite{AEO1977} when $1<p<\infty$, and from Capon's argument \cite{Capon1980} when $p=1$; alternatively, it follows for all $1\leq p<\infty$ from the factorisation theorem of Lechner and Speckhofer \cite[Theorem~3.1]{LechnerSpeckhofer2025}. Therefore \Cref{thm:transfer} applies to $E=\ell_p$ and to $E=L_p$, with $K=[0,1]$. It follows that $C([0,1],\ell_p)$ and $C([0,1],L_p)$ have the UPFP, and hence are primary by a standard application of Pe{\l}czy{\'n}ski's decomposition method.
\end{proof}

\subsection{Proof strategy and organisation}

It is useful to begin with the ideas in the scalar case, in other words, when $E = \mathbb{K}$, the scalar field. Given an operator $T\colon C(K)\to C(K)$, one composes with point evaluations and uses the Riesz representation theorem to represent each functional
\begin{equation*}
    \mu_t = \delta_tT\colon C(K)\longrightarrow \mathbb{K}
\end{equation*}
by the scalar measure $\mu_t$. The diagonal atom of this measure gives a scalar function, say $g(t) = \mu_t(\{t\})$. After a localisation step, the operator $T$ is then, on a suitable Cantor set, close to the multiplication operator $f\mapsto gf$. For a scalar multiplier, the primariness dichotomy is transparent: since
\begin{equation*}
    g(t)+(1-g(t))=1,
\end{equation*}
at each point one of $g(t)$ and $1-g(t)$ has modulus at least $1/2$. After passing to a suitable further set, this gives a multiplier bounded away from zero, and hence a local inverse needed to factor the identity through either $T$ or $I-T$.

The proof below follows this idea, but with the scalar multiplier replaced by an operator-valued multiplier. The Riesz representation theorem is replaced by the vector-valued representing measure theorem recalled in \Cref{sec:preliminaries}. Thus, for $T\in\mathscr{B}(C(K,E))$ and $t\in K$, we consider
\begin{equation*}
    \delta_tT\colon C(K,E)\longrightarrow E,\qquad f\longmapsto (Tf)(t).
\end{equation*}
Since $E$ contains no isomorphic copy of $c_0$, these operators are unconditionally converging, and \Cref{lem:no-c0-rows} applies. If $m_t$ denotes the representing measure of $\delta_tT$, the diagonal atom field is
\begin{equation*}
    A(t)=m_t(\{t\})\in\mathscr{B}(E).
\end{equation*}
This field is constructed and shown to be Borel in \Cref{sec:atom-field}, specifically in \Cref{lem:atom-borel}.

The localisation step is carried out using the local defect $\rho_O(t)$ from \Cref{def:local-defect}. This quantity measures the part of $(Tf)(t)$ not accounted for by the diagonal term $A(t)f(t)$ when $f$ is supported in $O$. By \Cref{lem:local-defect} and \Cref{prop:localization}, for every $\eta>0$ one can find an open set $O$ such that this off-diagonal contribution is less than $\eta$ on an uncountable Borel set. After passing to a Cantor subset and applying Lusin's theorem, $T$ is therefore close to the operator-valued multiplier
\begin{equation*}
    f(t)\longmapsto A(t)f(t)
\end{equation*}
on $C(D,E)$.

At this point the scalar alternatives $g$ and $1-g$ are replaced by
\begin{equation*}
    A(t)\qquad\text{and}\qquad I_E-A(t).
\end{equation*}
In the scalar case, the fact that one of $g(t)$ and $1-g(t)$ is bounded away from zero is used only to obtain the relevant factorisation alternative. In the Banach-space setting, this alternative is supplied directly by the UPFP of $E$: at each point $t$, the identity on $E$ factors uniformly through either $A(t)$ or $I_E-A(t)$.

The remaining problem is to make these pointwise choices coherently in $t$. This is the role of the descriptive set-theoretic selection argument in \Cref{lem:continuous-selection}, based on the tools collected in \Cref{lem:dst-input}. It produces a Cantor set $D$, a fixed side $\sigma\in\{0,1\}$, and strongly continuous fields $L$ and $R$ such that
\begin{equation*}
    L(t)A_\sigma(t)R(t)=I_E \qquad(t\in D).
\end{equation*}
By \Cref{lem:multipliers}, these fields define multiplication operators on $C(D,E)$.

Finally, in \Cref{sec:transfer}, the supported Dugundji extension from \Cref{lem:dugundji} is used to compress $T$ or $I-T$ to $C(D,E)$. The compression is a small perturbation of the corresponding operator-valued multiplier, and the perturbation is inverted by a Neumann-series argument. This gives the desired factorisation through one of $T$ and $I-T$.

\bigskip
\section{Notation and preliminary results}\label{sec:preliminaries}

We use standard notation and conventions, unless explicitly stated other\-wise. All Banach spaces are over the field $\mathbb{K}\in\{\R,\C\}$. By an \emph{operator}, we always mean a bounded linear map. For Banach spaces $X$ and $Y$, write $\mathscr{B}(X,Y)$ for the space of operators from $X$ to $Y$; when $X = Y$ we simply write this as $\mathscr{B}(X)$. The identity operator on $X$ is denoted by $I_X$. We denote by $\Delta$ the Cantor set.

We start by recalling the notion of the (uniform) primary factorisation property.

\begin{definition}\label{def:factor-through}
Let $S\colon Z\to W$ and $T\colon X\to Y$ be operators. We say that $S$ \emph{factors through} $T$ with constant $C$ if there are operators $V\colon Z\to X$ and $U\colon Y\to W$ such that
\begin{equation*}
    UTV=S \qquad\text{and}\qquad \norm{U}\norm{V}\leq C.
\end{equation*}
\end{definition}

\begin{definition}\label{def:upfp}
Let $X$ be a Banach space.
\begin{enumerate}[label=(\alph*)]
    \item $X$ has the \emph{primary factorisation property} (PFP) if, for every $T\in\mathscr{B}(X)$, the identity $I_X$ factors through either $T$ or $I_X-T$.
    \item For $C\geq 1$, $X$ has the \emph{$C$-primary factorisation property} if the factorisation in (a) can always be chosen with constant at most $C$.
    \item $X$ has the \emph{uniform primary factorisation property} (UPFP) if it has the $C$-primary factorisation property for some $C\geq 1$.
\end{enumerate}
\end{definition}

Throughout the paper, we shall use the fact that the UPFP is invariant under isomorphisms; see, for instance, \cite[Proposition~2.5]{AcuavivaKania2026}. In particular, by Millutin's theorem \cite{Milutin1966} and standard results of the injective tensor products, \cite[Section 3.2]{Ryan2002}, we have
\begin{equation*}
    C(K,E)= C(K)\widehat{\otimes}_{\varepsilon} E \simeq C(\Delta) \widehat{\otimes}_{\varepsilon} E = C(\Delta, E),
\end{equation*}
and thus it is enough to prove \Cref{thm:transfer} in the case $K = \Delta$. 

We shall need the following formulation of the Dugundji extension theorem.

\begin{lemma}[Supported Dugundji extension]\label{lem:dugundji}
Let $K$ be compact metrizable, let $D\subseteq K$ be closed, and let $E$ be a Banach space. There is a linear extension operator
\begin{equation*}
    \mathcal{E}_D\colon C(D,E)\longrightarrow C(K,E)
\end{equation*}
with $\norm{\mathcal{E}_D}\leq 1$ and $\restr{\mathcal{E}_Df}{D}=f$. If $O\subseteq K$ is open and $D\subseteq O$, then there is an operator
\begin{equation*}
    J_D\colon C(D,E)\longrightarrow C(K,E)
\end{equation*}
such that
\begin{equation*}
    \norm{J_D}\leq 1,\qquad \restr{J_Df}{D}=f,\qquad \restr{J_Df}{K\setminus O}=0.
\end{equation*}
\end{lemma}

\begin{proof}
Use Dugundji's linear extension theorem \cite{Dugundji1951} to obtain $\mathcal{E}_D$. For the supported version, choose $h\in C(K,[0,1])$ with $h=1$ on $D$ and $h=0$ on $K\setminus O$, and set $J_Df=h\mathcal{E}_Df$.
\end{proof}

\subsection{Bounded strong operator balls}

For $M\geq 0$, set
\begin{equation*}
    \mathscr{B}_M(E)=\{R\in\mathscr{B}(E):\norm{R}\leq M\},
\end{equation*}
equipped with the strong operator topology. We will need the following presumably standard fact.

\begin{lemma}[Bounded strong operator balls]\label{lem:strong-balls}
Let $E$ be separable. Then $\mathscr{B}_M(E)$ is Polish in the strong operator topology. On norm-bounded operator sets, the evaluation map
\begin{equation*}
    (R,x)\longmapsto Rx
\end{equation*}
is jointly continuous, and multiplication
\begin{equation*}
    (R,Q)\longmapsto RQ
\end{equation*}
is jointly continuous whenever all operators range over fixed norm-bounded sets.
\end{lemma}

\begin{proof}
Choose a norm-dense sequence $(x_j)_{j \in \N}$ in the unit ball of $E$ and use the standard metric
\begin{equation*}
    d(R,Q)=\sum_{j=1}^{\infty}2^{-j} \frac{\norm{(R-Q)x_j}}{1+\norm{(R-Q)x_j}}.
\end{equation*}
Completeness follows from the uniform bound $M$, and separability follows by embedding $\mathscr{B}_M(E)$ into a closed subset of $E^\N$. Joint evaluation and multiplication are obtained by the estimates
\begin{equation*}
    \norm{R_nx_n-Rx} \leq M\norm{x_n-x}+\norm{(R_n-R)x}
\end{equation*}
and
\begin{equation*}
    \norm{R_nQ_nx-RQx} \leq \norm{R_n}\norm{(Q_n-Q)x}+\norm{(R_n-R)Qx}.
\end{equation*}
\end{proof}

We shall need the following vector-valued analogue of scalar multiplication operators.  In the scalar case, a bounded continuous function $g\in C(D)$ defines the operator
\begin{equation*}
    f\longmapsto gf,\qquad (gf)(t)=g(t)f(t).
\end{equation*}
Here, the scalar multiplier $g(t)$ is replaced by an operator $R(t)\in\mathscr{B}(E)$, depending strongly continuously on $t$.

\begin{lemma}[Strongly continuous multipliers]\label{lem:multipliers}
Let $D$ be compact and let $R\colon D\to\mathscr{B}(E)$ be strongly continuous and uniformly bounded. Then
\begin{equation*}
    (M_Rf)(t)=R(t)f(t) \qquad(t\in D)
\end{equation*}
defines an operator $M_R\in\mathscr{B}(C(D,E))$, and
\begin{equation*}
    \norm{M_R}=\sup_{t\in D}\norm{R(t)}.
\end{equation*}
\end{lemma}

\begin{proof}
Use norm-continuity of $f$ and strong continuity of $R$ at each $t_0\in D$:
\begin{equation*}
    \norm{R(t)f(t)-R(t_0)f(t_0)} \leq \norm{R(t)}\norm{f(t)-f(t_0)} +\norm{(R(t)-R(t_0))f(t_0)}.
\end{equation*}
The upper norm estimate is immediate, and the reverse estimate follows by testing on constant functions.
\end{proof}

\subsection{Representing measures}

We shall use the representing measure theorem for operators on vector-valued continuous function spaces. Let $L$ be a compact Hausdorff space, and let $\mathcal{B}(L)$ denote its Borel $\sigma$-algebra. For an operator
\begin{equation*}
    R\colon C(L,E)\longrightarrow F
\end{equation*}
by \cite[Theorem~2.2]{BrooksLewis1974} there is a representing measure
\begin{equation*}
    m\colon \mathcal{B}(L)\longrightarrow \mathscr{B}(E,F^{**})
\end{equation*}
of bounded semivariation such that
\begin{equation*}
    \kappa_FRf=\int_L f(s)\,dm(s) \qquad(f\in C(L,E)),
\end{equation*}
where $\kappa_F\colon F\to F^{**}$ is the canonical embedding. 

We shall also need the following, which is just \cite[Lemma~3.1]{BrooksLewis1974} adapted to our current notation.

\begin{lemma}[Strong boundedness and semivariation]\label{lem:strong-bounded-semivar}
Let $m\colon \mathcal{B}(L)\to\mathscr{B}(E,F)$ be a representing measure, and write $\widetilde m(B)$ for the semivariation of $m$ on $B\in\mathcal{B}(L)$. Then $m$ is strongly bounded if and only if
\begin{equation*}
    B_n\downarrow\varnothing
    \quad\Longrightarrow\quad
    \widetilde m(B_n)\longrightarrow 0.
\end{equation*}
\end{lemma}

We shall use the representing measure theorem through the following unconditionally converging result.

\begin{lemma}[Representing measures for no-$c_0$ ranges]\label{lem:no-c0-rows}
Let $E$ contain no isomorphic copy of $c_0$, let $L$ be compact Hausdorff, and let
\begin{equation*}
    R\colon C(L,E)\longrightarrow E
\end{equation*}
be bounded. Then $R$ is unconditionally converging. Consequently, its representing measure takes values in $\mathscr{B}(E)$ and is strongly bounded.
\end{lemma}

\begin{proof}
Let $(f_n)_{n\in\N}$ be a sequence in $C(L,E)$ such that
\begin{equation*}
    \sum_{n=1}^{\infty} f_n
\end{equation*}
is weakly unconditionally Cauchy. Then
\begin{equation*}
    \sum_{n=1}^{\infty} Rf_n
\end{equation*}
is weakly unconditionally Cauchy in $E$, since bounded operators preserve weakly unconditionally Cauchy series.

By a classical result of Bessaga and Pe{\l}czy{\'n}ski \cite[Theorem~5]{BessagaPelczynski1958}, a Banach space contains an isomorphic copy of $c_0$ if and only if it admits a weakly unconditionally Cauchy series which is not unconditionally convergent. Since $E$ contains no isomorphic copy of $c_0$, it follows that
\begin{equation*}
    \sum_{n=1}^{\infty} Rf_n
\end{equation*}
is unconditionally convergent. Hence $R$ is unconditionally converging.

Let
\begin{equation*}
    m\colon \mathcal{B}(L)\longrightarrow \mathscr{B}(E,E^{**})
\end{equation*}
be the representing measure of $R$. By \cite[Theorem~2.1]{Saab1984}, or equivalently by Dobrakov's theorem as recalled immediately before that theorem, the representing measure of an unconditionally converging operator on $C(L,E)$ has semivariation continuous at $0$ and, after identifying $E$ with its canonical image in $E^{**}$, satisfies
\begin{equation*}
    m(B)\in \mathscr{B}(E)
    \qquad(B\in\mathcal{B}(L)).
\end{equation*}
Thus $m$ takes its values in $\mathscr{B}(E)$ and, by \Cref{lem:strong-bounded-semivar}, is strongly bounded.
\end{proof}

Lastly, we need the following standard result.

\begin{lemma}[Semivariation estimates]\label{lem:semivar-estimates}
Let $R\colon C(L,E)\to F$ be bounded, and let $m$ be its representing measure. Then, for every $B\in\mathcal{B}(L)$,
\begin{equation*}
    \norm{m(B)}\leq \widetilde m(B)\leq \widetilde m(L)=\norm{R}.
\end{equation*}
Moreover, whenever $g\colon L\to E$ is bounded Borel, has relatively compact range, and vanishes off $B$, one has
\begin{equation}\label{eq:semivariation-integral}
    \norm{\int_B g\,dm}\leq \widetilde m(B)\norm{g}_{\infty}.
\end{equation}
\end{lemma}

\begin{proof}
The inequality $\norm{m(B)}\leq\widetilde m(B)$ follows directly from the definition of semivariation by testing the one-set partition of $B$. The inequality $\widetilde m(B)\leq\widetilde m(L)$ is immediate from monotonicity of semivariation, and the equality $\widetilde m(L)=\norm{R}$ is the norm-preserving part of the representing measure theorem.

For the integral estimate, first suppose that $g$ is simple and vanishes off $B$, say
\begin{equation*}
    g=\sum_{i=1}^n \chi_{B_i}x_i,
\end{equation*}
where $B_i\in\mathcal{B}(L)$ are disjoint subsets of $B$. Then
\begin{equation*}
    \norm{\int_B g\,dm}=\norm{\sum_{i=1}^n m(B_i)x_i} \leq \widetilde m(B)\max_{1\leq i\leq n}\norm{x_i}\leq \widetilde m(B)\norm{g}_{\infty}.
\end{equation*}
The general case follows by uniformly approximating $g$ by bounded Borel simple functions with values in the relatively compact range of $g$ and passing to the limit in the representing integral.
\end{proof}

\subsection{Descriptive set theory} 

We will need to use some key results from descriptive set theory. From now on, we will say that a topological space $P$ is a \emph{Cantor set} if it is homeomorphic to the Cantor set $\Delta$.

\begin{lemma}\label{lem:dst-input}
Let $X$ and $Y$ be Polish spaces.
\begin{enumerate}[label=\textup{(\roman*)}, ref=\textup{(\roman*)}]
    \item\label{item:dst-cantor}
    Every uncountable analytic subset of $X$ contains a Cantor set.

    \item\label{item:dst-jvn}
    If $A\subseteq X\times Y$ is analytic and every section
    \begin{equation*}
        A_x=\{y\in Y:(x,y)\in A\}
    \end{equation*}
    is non-empty, then there is a universally measurable map $\phi\colon X\to Y$ such that
    \begin{equation*}
        (x,\phi(x))\in A \qquad(x\in X).
    \end{equation*}

    \item\label{item:dst-lusin}
    If $\mu$ is a finite Borel measure on $X$ and $\phi\colon X\to Y$ is $\mu$-measurable, then for every $\varepsilon>0$ there is a compact set $C\subseteq X$ such that
    \begin{equation*}
        \mu(X\setminus C)<\varepsilon
    \end{equation*}
    and $\phi|_C$ is continuous.
\end{enumerate}
\end{lemma}

\begin{proof}
\ref{item:dst-cantor} is the perfect set theorem for analytic sets; see \cite[Theorem~29.1, p.~226]{Kechris1995}. Since every Borel set is analytic, it also applies to uncountable Borel sets. \ref{item:dst-jvn} is the Jankov--von Neumann uniformization theorem in its universally measurable form; see \cite[Theorem~29.9, p.~227]{Kechris1995}. Finally, \ref{item:dst-lusin} is Lusin's theorem; see \cite[Theorem~17.12, p.~108]{Kechris1995}.
\end{proof}

We now prove the following result.

\begin{lemma}[Continuous factor selection on a Cantor set]\label{lem:continuous-selection}
Let $P$ be a Cantor set, $E$ be a separable Banach space, and
\begin{equation*}
    A\colon P\longrightarrow \mathscr{B}_M(E)
\end{equation*}
be strongly continuous. Fix $r>0$. Suppose that for every $t\in P$ there are $\sigma(t)\in\{0,1\}$ and operators $L_t,R_t\in \mathscr{B}_r(E)$ such that
\begin{equation*}
    L_tA_{\sigma(t)}(t)R_t=I_E, \qquad A_0(t)=A(t),\quad A_1(t)=I_E-A(t).
\end{equation*}
Then there are a Cantor set $D\subseteq P$, a fixed $\sigma\in\{0,1\}$, and strongly continuous maps
\begin{equation*}
    L,R\colon D\longrightarrow \mathscr{B}_r(E)
\end{equation*}
such that
\begin{equation*}
    L(t)A_\sigma(t)R(t)=I_E \qquad(t\in D).
\end{equation*}
\end{lemma}

\begin{proof}
For $\sigma\in\{0,1\}$, define
\begin{equation*}
    \Gamma_\sigma=\{(t,L,R)\in P\times\mathscr{B}_r(E)^2:LA_\sigma(t)R=I_E\}.
\end{equation*}
By \Cref{lem:strong-balls}, the balls $\mathscr{B}_r(E)$ are Polish for the strong operator topology, and the maps involved in the definition of $\Gamma_\sigma$ are continuous. Hence, each $\Gamma_\sigma$ is closed. The projections of $\Gamma_0$ and $\Gamma_1$ onto $P$ are analytic, and they cover $P$. Therefore one of them is uncountable, and by \Cref{lem:dst-input} \ref{item:dst-cantor} it contains a Cantor set $Q$.

Fix this $\sigma$. By \Cref{lem:dst-input} \ref{item:dst-jvn}, applied to the relation $\Gamma_\sigma\cap(Q\times\mathscr{B}_r(E)^2)$, there are universally measurable maps
\begin{equation*}
    L,R\colon Q\longrightarrow \mathscr{B}_r(E)
\end{equation*}
such that
\begin{equation*}
    L(t)A_\sigma(t)R(t)=I_E \qquad(t\in Q).
\end{equation*}
Choose a nonatomic Borel probability measure on $Q$ with full support. By \Cref{lem:dst-input} \ref{item:dst-lusin}, there is a compact set $C\subseteq Q$ of positive measure on which both $L$ and $R$ are strongly continuous. Since the measure is nonatomic, $C$ is uncountable. By \Cref{lem:dst-input} \ref{item:dst-cantor}, $C$ contains a Cantor set $D$. Restricting $L$ and $R$ to $D$ gives the desired strongly continuous factor fields. 
\end{proof}

\begin{remark}
    The preceding lemma is the point at which the descriptive set-theoretic machinery enters the proof. Its role is to turn pointwise factorisations of the operators $A(t)$ or $I_E-A(t)$ into factorisations whose left and right factors vary strongly continuously on a Cantor subset. This continuity is essential in the sequel, since it allows the factor fields to define bounded multiplication operators on $C(D,E)$.
\end{remark}

\bigskip
\section{The diagonal atom field}\label{sec:atom-field}

Let $K$ be an uncountable compact metrizable space, and let $E$ be a separable Banach space containing no isomorphic copy of $c_0$. Throughout this section, fix an operator
\begin{equation*}
    T\in\mathscr{B}(C(K,E)).
\end{equation*}
For $t\in K$, define the row operator
\begin{equation*}
    T_t=\delta_tT\colon C(K,E)\longrightarrow E, \qquad T_tf=(Tf)(t).
\end{equation*}

By \Cref{lem:no-c0-rows}, $T_t$ is unconditionally converging and its representing measure $m_t$ takes its values in $\mathscr{B}(E)$ and is strongly bounded. Define the diagonal atom
\begin{equation*}
    A(t)=m_t(\{t\})\in\mathscr{B}(E).
\end{equation*}
By \Cref{lem:semivar-estimates},
\begin{equation*}
    \norm{A(t)}\leq \widetilde m_t(\{t\})\leq \widetilde m_t(K)=\norm{T_t}\leq\norm{T}.
\end{equation*}

\begin{lemma}[Borel atom field]\label{lem:atom-borel}
The map
\begin{equation*}
    A\colon K\longrightarrow \mathscr{B}_{\norm{T}}(E)
\end{equation*}
is Borel for the strong operator topology.
\end{lemma}
\begin{proof}
Fix a compatible metric $d$ on $K$, and put
\begin{equation*}
    \varphi_n(t,s)=\max\{1-nd(t,s),0\} \qquad(t,s\in K).
\end{equation*}
For each $n\in\N$, define
\begin{equation*}
    A_n\colon K\longrightarrow \mathscr{B}(E)
\end{equation*}
by
\begin{equation*}
    A_n(t)x=T(\varphi_n(t,\cdot)x)(t) \qquad(t\in K,\ x\in E).
\end{equation*}
We first check that $A_n(t)\in\mathscr{B}(E)$ and that $t\mapsto A_n(t)$ is strongly continuous. Linearity of $A_n(t)$ in $x$ is immediate. Moreover,
\begin{equation*}
    \norm{A_n(t)x}\leq \norm{T}\norm{\varphi_n(t,\cdot)x}_{\infty}\leq \norm{T}\norm{x},
\end{equation*}
so $A_n(t)\in\mathscr{B}(E)$ and $\norm{A_n(t)}\leq\norm{T}$.

Fix $x\in E$. The map
\begin{equation*}
    t\longmapsto \varphi_n(t,\cdot)x
\end{equation*}
is norm-continuous from $K$ into $C(K,E)$. Hence the map
\begin{equation*}
    g_n\colon K\longrightarrow C(K,E)
\end{equation*}
defined by
\begin{equation*}
    g_n(t)=T(\varphi_n(t,\cdot)x) \qquad(t\in K)
\end{equation*}
is norm-continuous. Equivalently,
\begin{equation*}
    g_n(t)(s)=T(\varphi_n(t,\cdot)x)(s) \qquad(t,s\in K).
\end{equation*}
Since
\begin{equation*}
    A_n(t)x=g_n(t)(t),
\end{equation*}
and the diagonal evaluation map
\begin{equation*}
    K\times C(K,E)\longrightarrow E,\qquad (t,g)\longmapsto g(t),
\end{equation*}
is continuous, the map
\begin{equation*}
    t\longmapsto A_n(t)x
\end{equation*}
is continuous. Since $x\in E$ was arbitrary, $A_n$ is strongly continuous.

For fixed $t\in K$ and $x\in E$, the representing measure formula gives
\begin{equation*}
    A_n(t)x=\int_K \varphi_n(t,s)x\,dm_t(s).
\end{equation*}
Since $\varphi_n(t,t)=1$, the contribution of the atom $\{t\}$ is
\begin{equation*}
    \int_{\{t\}}\varphi_n(t,s)x\,dm_t(s)=m_t(\{t\})x=A(t)x.
\end{equation*}
Thus
\begin{equation*}
    A_n(t)x-A(t)x=\int_{K\setminus\{t\}}\varphi_n(t,s)x\,dm_t(s).
\end{equation*}
Since $\varphi_n(t,s)=0$ whenever $d(s,t)\geq n^{-1}$, \Cref{lem:semivar-estimates} gives
\begin{equation*}
    \norm{A_n(t)x-A(t)x}\leq \widetilde m_t(\{s:0<d(s,t)\leq n^{-1}\})\norm{x}.
\end{equation*}
The sets
\begin{equation*}
    \{s:0<d(s,t)\leq n^{-1}\}
\end{equation*}
decrease to the empty set. Since $m_t$ is strongly bounded, \Cref{lem:strong-bounded-semivar} gives
\begin{equation*}
    A_n(t)x\longrightarrow A(t)x.
\end{equation*}
Thus $A$ is the pointwise strong limit of the Borel maps $A_n$. Since $\mathscr{B}_{\norm{T}}(E)$ is metrizable for the strong operator topology by \Cref{lem:strong-balls}, the map $A$ is Borel.
\end{proof}

The next definition measures the local off-diagonal part of $T$. For the row operator $\delta_tT$, the atom $A(t)=m_t(\{t\})$ records the contribution coming from the value of a function at $t$. Thus $A(t)f(t)$ is the diagonal part of $(Tf)(t)$. If $f$ is supported in an open set $O$ containing $t$, then the difference
\begin{equation*}
    (Tf)(t)-A(t)f(t)
\end{equation*}
comes only from the part of the representing measure on $O\setminus\{t\}$.

\begin{definition}[Local defect]\label{def:local-defect}
For an open set $O\subseteq K$ and $t\in O$, define
\begin{equation*}
    \rho_O(t)=\sup\{\norm{(Tf)(t)-A(t)f(t)}:\norm{f}\leq 1,\ \restr{f}{K\setminus O}=0\}.
\end{equation*}
\end{definition}

We shall use the following elementary measurability fact.

\begin{lemma}[Borel evaluation]\label{lem:borel-evaluation}
Let $X$ be a measurable space, let $E$ be a separable Banach space, and let
\begin{equation*}
    S\colon X\longrightarrow \mathscr{B}(E)
\end{equation*}
be strongly Borel and uniformly bounded. If $u\colon X\to E$ is Borel, then
\begin{equation*}
    x\longmapsto S(x)u(x)
\end{equation*}
is Borel.
\end{lemma}

\begin{proof}
First suppose that $u$ is simple, say
\begin{equation*}
    u=\sum_{i=1}^n \mathbf{1}_{B_i}x_i.
\end{equation*}
Then
\begin{equation*}
    S(x)u(x)=\sum_{i=1}^n \mathbf{1}_{B_i}(x)S(x)x_i,
\end{equation*}
which is Borel by strong Borelness of $S$.

For general Borel $u$, since $E$ is separable, there are simple Borel maps $u_n\colon X\to E$ such that $u_n(x)\to u(x)$ for every $x\in X$. By uniform boundedness of $S$,
\begin{equation*}
    \norm{S(x)u_n(x)-S(x)u(x)}\leq \sup_{y\in X}\norm{S(y)}\,\norm{u_n(x)-u(x)}.
\end{equation*}
Thus $S(x)u_n(x)\to S(x)u(x)$ pointwise. Since $E$ is metric, pointwise limits of Borel maps into $E$ are Borel.
\end{proof}

The following lemma makes the preceding intuition precise: $\rho_O(t)$ is controlled by the semivariation of the off-diagonal part of $m_t$ on $O$.

\begin{lemma}[Local defect estimate]\label{lem:local-defect}
For every open $O\subseteq K$, the function
\begin{equation*}
    \rho_O\colon O\longrightarrow [0,\infty)
\end{equation*}
is Borel. Moreover,
\begin{equation*}
    \rho_O(t)\leq \widetilde m_t(O\setminus\{t\}) \qquad(t\in O).
\end{equation*}
\end{lemma}

\begin{proof}
Let
\begin{equation*}
    Z_O=\{f\in C(K,E):\restr{f}{K\setminus O}=0\}.
\end{equation*}
Since $K$ is metrizable and $E$ is separable, the space $Z_O$ is separable. Choose a norm-dense sequence $(f_j)_{j\in\N}$ in the unit ball of $Z_O$. Then
\begin{equation*}
    \rho_O(t)=\sup_{j\in\N}\norm{(Tf_j)(t)-A(t)f_j(t)}.
\end{equation*}
For each $j$, the map $t\mapsto (Tf_j)(t)$ is continuous. Also, $f_j\colon O\to E$ is Borel, while $A\colon O\to\mathscr{B}(E)$ is strongly Borel and uniformly bounded. Hence, by \Cref{lem:borel-evaluation}, the map
\begin{equation*}
    t\longmapsto A(t)f_j(t)
\end{equation*}
is Borel. Since $\rho_O$ is the pointwise supremum of a countable family of Borel functions, $\rho_O$ is Borel.

If $f\in Z_O$, then the representing measure formula gives
\begin{equation*}
    (Tf)(t)-A(t)f(t)=\int_{O\setminus\{t\}}f(s)\,dm_t(s).
\end{equation*}
By \eqref{eq:semivariation-integral} we have
\begin{equation*}
    \norm{(Tf)(t)-A(t)f(t)}\leq \widetilde m_t(O\setminus\{t\})\norm{f}.
\end{equation*}
Taking the supremum over all $f\in Z_O$ with $\norm{f}\leq 1$ gives the desired estimate.
\end{proof}

The next proposition is the localisation step. It says that the local off-diagonal contribution can be made smaller than any prescribed $\eta>0$ on an uncountable set of points, after choosing a suitable open set $O$.

\begin{proposition}[Localization on an uncountable Borel set]\label{prop:localization}
Let $\eta>0$. There is an open set $O\subseteq K$ such that
\begin{equation*}
    B_O=\{t\in O:\rho_O(t)<\eta\}
\end{equation*}
is an uncountable Borel subset of $K$.
\end{proposition}

\begin{proof}
Let $\mathcal U$ be a countable base for $K$. For each $t\in K$, choose decreasing open metric balls $V_n$ centered at $t$ with
\begin{equation*}
    \bigcap_{n\in\N} V_n=\{t\}.
\end{equation*}
By \Cref{lem:local-defect},
\begin{equation*}
    \rho_{V_n}(t)\leq \widetilde m_t(V_n\setminus\{t\}).
\end{equation*}
Since the sets $V_n\setminus\{t\}$ decrease to the empty set and $m_t$ is strongly bounded, \Cref{lem:strong-bounded-semivar} gives
\begin{equation*}
    \rho_{V_n}(t)\longrightarrow 0.
\end{equation*}
Choose $n$ with $\rho_{V_n}(t)<\eta$, and then choose $O\in\mathcal U$ with $t\in O\subseteq V_n$. Since $O\subseteq V_n$, the definition gives
\begin{equation*}
    \rho_O(t)\leq \rho_{V_n}(t)<\eta.
\end{equation*}
For each $O\in\mathcal U$, define
\begin{equation*}
    B_O=\{t\in O:\rho_O(t)<\eta\}.
\end{equation*}
The preceding argument shows that
\begin{equation*}
    K=\bigcup_{O\in\mathcal U}B_O.
\end{equation*}
Each $B_O$ is Borel by \Cref{lem:local-defect}. Since $K$ is uncountable and $\mathcal U$ is countable, at least one $B_O$ is uncountable.
\end{proof}

\bigskip
\section{Proof of \Cref{thm:transfer}}\label{sec:transfer}

We now prove \Cref{thm:transfer}. By the comments in \Cref{sec:preliminaries}, it is enough to prove the result when $K$ is the Cantor set.

\begin{proof}[Proof of \Cref{thm:transfer}]
Assume that $E$ has the $C$-primary factorisation property. The case $E=\{0\}$ is immediate, so assume $E\neq\{0\}$.

Fix $T\in\mathscr{B}(C(K,E))$, and put
\begin{equation*}
    \eta=\frac{1}{4C}.
\end{equation*}
Apply \Cref{prop:localization} to obtain an open set $O\subseteq K$ for which
\begin{equation*}
    B_O=\{t\in O:\rho_O(t)<\eta\}
\end{equation*}
is an uncountable Borel subset of $K$. By \Cref{lem:dst-input} \ref{item:dst-cantor}, there is a Cantor set $P\subseteq B_O$.

By \Cref{lem:atom-borel}, the atom field
\begin{equation*}
    A\colon P\longrightarrow \mathscr{B}_{\norm{T}}(E)
\end{equation*}
is strongly Borel. Choose a nonatomic Borel probability measure on $P$ with full support. By \Cref{lem:dst-input} \ref{item:dst-lusin}, and then by \Cref{lem:dst-input} \ref{item:dst-cantor}, after replacing $P$ by a smaller Cantor set we may assume that
\begin{equation*}
    A\colon P\longrightarrow \mathscr{B}_{\norm{T}}(E)
\end{equation*}
is strongly continuous.

For every $t\in P$, apply the $C$-primary factorisation property of $E$ to $A(t)$. Thus, for one of
\begin{equation*}
    A_0(t)=A(t), \qquad A_1(t)=I_E-A(t),
\end{equation*}
there are $L_t,R_t\in\mathscr{B}(E)$ such that
\begin{equation*}
    L_tA_{\sigma(t)}(t)R_t=I_E, \qquad \norm{L_t}\norm{R_t}\leq C.
\end{equation*}
Rescaling the factors, we may assume that
\begin{equation*}
    \norm{L_t}\leq \sqrt C \quad \text{and} \quad \norm{R_t}\leq \sqrt C.
\end{equation*}
By \Cref{lem:continuous-selection}, there are a Cantor set $D\subseteq P$, a fixed $\sigma\in\{0,1\}$, and strongly continuous maps
\begin{equation*}
    L,R\colon D\longrightarrow \mathscr{B}_{\sqrt C}(E)
\end{equation*}
such that
\begin{equation*}
    L(t)A_\sigma(t)R(t)=I_E \qquad(t\in D).
\end{equation*}

Let
\begin{equation*}
    Q_D\colon C(K,E)\longrightarrow C(D,E)
\end{equation*}
be the restriction operator from $K$ to $D$. Since $D\subseteq O$, \Cref{lem:dugundji} gives a supported extension
\begin{equation*}
    J_D\colon C(D,E)\longrightarrow C(K,E)
\end{equation*}
such that
\begin{equation*}
    \norm{J_D}\leq 1,\qquad Q_DJ_D=I_{C(D,E)},\qquad \restr{J_Df}{K\setminus O}=0.
\end{equation*}
Put
\begin{equation*}
    T_0=T,\qquad T_1=I_{C(K,E)}-T,\qquad S_\sigma=Q_DT_\sigma J_D.
\end{equation*}

Since $D\subseteq B_O$, for every $f\in C(D,E)$ with $\norm{f}\leq 1$ and every $t\in D$,
\begin{equation*}
    \norm{(Q_DTJ_Df)(t)-A(t)f(t)}=\norm{(TJ_Df)(t)-A(t)(J_Df)(t)}<\eta.
\end{equation*}

By \Cref{lem:multipliers}, the strongly continuous fields $A_0$ and $A_1$ define multipliers $M_{A_0}$ and $M_{A_1}$ on $C(D,E)$. Since $A_0=A$ and $A_1=I_E-A$, the preceding estimate and the identity $Q_DJ_D=I_{C(D,E)}$ give, for $\sigma \in \{0,1\}$,
\begin{equation*}
    \norm{S_\sigma-M_{A_\sigma}}\leq\eta.
\end{equation*}

By \Cref{lem:multipliers}, the strongly continuous fields $L$ and $R$ define multipliers $M_L$ and $M_R$. Since
\begin{equation*}
    L(t)A_\sigma(t)R(t)=I_E \qquad(t\in D),
\end{equation*}
we have
\begin{equation*}
    M_LM_{A_\sigma}M_R=I_{C(D,E)}.
\end{equation*}
Therefore
\begin{equation*}
    \norm{M_LS_\sigma M_R-I_{C(D,E)}}\leq \norm{M_L}\norm{S_\sigma-M_{A_\sigma}}\norm{M_R}\leq C\eta=\frac14.
\end{equation*}
Set
\begin{equation*}
    G=M_LS_\sigma M_R.
\end{equation*}
Then $G$ is invertible and $\norm{G^{-1}}\leq 4/3$. Hence
\begin{equation*}
    I_{C(D,E)}=G^{-1}M_LQ_D\,T_\sigma\,J_DM_R.
\end{equation*}
Thus the identity on $C(D,E)$ factors through $T_\sigma$ with constant at most $4C/3$.

Since $K$ and $D$ are both Cantor sets, choose a homeomorphism $h\colon K\to D$. The induced composition operator
\begin{equation*}
    W\colon C(K,E)\longrightarrow C(D,E),\qquad Wf=f\circ h^{-1},
\end{equation*}
is an isometric isomorphism. Multiplying the preceding factorisation by $W^{-1}$ on the left and by $W$ on the right gives
\begin{equation*}
    I_{C(K,E)}=W^{-1}G^{-1}M_LQ_D\,T_\sigma\,J_DM_RW.
\end{equation*}
Thus $I_{C(K,E)}$ factors through $T_\sigma$ with constant at most
\begin{equation*}
    \frac{4C}{3}.
\end{equation*}
Since $T\in\mathscr{B}(C(K,E))$ was arbitrary, $C(K,E)$ has the $(4C/3)$-primary factorisation property.
\end{proof}
\bigskip

\noindent\textbf{Acknowledgements.} This paper forms part of the first-named author’s PhD research at Lancaster University, conducted under the supervision of Professor N. J. Laustsen. He acknowledges with thanks the funding from the EPSRC (grant number EP/W524438/1) that has supported his studies. \medskip

\noindent\textbf{AI usage statement.} Large language models, in particular OpenAI's ChatGPT 5.5 Pro, were used during the exploratory stages of this work. It suggested the use of descriptive set-theoretic selection tools, including the Jankov--von Neumann theorem and Lusin's theorem, which became an essential ingredient in the final proof. The authors take full responsibility for the mathematical correctness of the paper. \medskip

For the purpose of open access, the author has applied a Creative Commons Attribution (CC BY) licence to any Author Accepted Manuscript version arising. \medskip

\noindent\textbf{Data availability.} No data was used for the research described in the article.

\end{document}